\documentclass[11pt]{amsart}

\usepackage{amsthm, amsfonts, amssymb, amscd}
\usepackage{tikz-cd}
\usepackage[pagebackref,colorlinks]{hyperref}

\theoremstyle{definition}
\newtheorem{ntn}{Notation}[section]

\theoremstyle{plain}
\newtheorem{cor}[ntn]{Corollary}
\newtheorem{lem}[ntn]{Lemma}
\newtheorem{prp}[ntn]{Proposition}
\newtheorem{thm}[ntn]{Theorem}

\theoremstyle{remark}
\newtheorem{rem}[ntn]{Remark}
\newtheorem{exa}[ntn]{Example}

\numberwithin{equation}{section}

\newcommand{\z}{\mathbb{Z}}
\newcommand{\q}{\mathbb{Q}}
\newcommand{\R}{\mathbb{R}}

\newcommand{\A}{\mathcal{A}}

\newcommand{\RR}{\mathcal{R}}

\newcommand{\mt}{\mapsto}
\newcommand{\arr}{\rightarrow}
\newcommand{\larr}{\longrightarrow}
\newcommand{\harr}{\hookrightarrow}
\newcommand{\se}{\subseteq}
\newcommand{\two}{\twoheadrightarrow}
\newcommand{\tail}{\rightarrowtail}

\newcommand{\inc}{{\rm inc}}
\newcommand{\coker}{{\rm coker}}
\newcommand{\im}{{\rm im}}
\newcommand{\Hom}{{\rm Hom}}

\newcommand{\id}{{\rm id}}
\newcommand{\tor}{{{\rm Tor}}}

\newcommand{\SL}{{\rm SL}}
\newcommand{\PSL}{{\rm PSL}}

\newcommand{\tors}{{{\rm Tor}_1^{\z}}}

\newcommand\blfootnote[1]{%
  \begingroup
  \renewcommand\thefootnote{}\footnote{#1}%
  \addtocounter{footnote}{-1}%
  \endgroup
}

\newtheoremstyle{athm}
  {}
  {}
  {\itshape}
  {}
  {\scshape}
  {}
  {.5em}
  {\thmnote{#3}}
\theoremstyle{athm}

\begin{document}
\title[ $H_3$ of stem extensions and Whitehead's quadratic functor]{The third homology of stem-extensions and 
Whitehead's quadratic functor}
\author{B.  Mirzaii}
\author{F. Y. Mokari}
\author{D. C. Ordinola}

\begin{abstract}
%Let  $A\tail G\two Q$ be a stem-extension and let $\rho: A\times G\arr G$ be the multiplication map.
%We show that there is a natural map
%\[
%\varphi: H_1(\Sigma_2^\epsilon, \tors({}_{2^\infty}A,{}_{2^\infty}A))\arr H_3(G,\z)/\rho_\ast(A \otimes_\z H_2(G,\z)) 
%\]
%such that, the image of $\varphi$ coincides with the image of the natural map $H_3(A,\z)\arr H_3(G,\z)/\rho_\ast(A \otimes_\z H_2(G,\z))$. An
%important tool used here is Whitehead's quadratic functor $\Gamma$.
%%, which is a functor from the category of abelian groups to itself.
%As part of our proof of the main result, we give a precise homological description  of the kernel of the natural map
%$\Gamma(A) \arr A\otimes_\z A$, $\gamma(a)\mapsto a\otimes a$.

Let  $A \rightarrowtail G\twoheadrightarrow Q$ be a stem-extension and let $\rho: A\times G\to G$ be the multiplication map.
We show that there is a natural map
\[
\varphi: H_1(\Sigma_2^\epsilon, {\rm Tor}_1^{\mathbb{Z}}({}_{2^\infty}A,{}_{2^\infty}A))\to 
H_3(G,\mathbb{Z})/\rho_\ast(A \otimes_{\mathbb{Z}} H_2(G,\mathbb{Z})) 
\]
such that, the image of $\varphi$ coincides with the image of the natural map 
$H_3(A,\mathbb{Z})\to H_3(G,\mathbb{Z})/\rho_\ast(A \otimes_{\mathbb{Z}} 
H_2(G,\mathbb{Z}))$. An
important tool used here is Whitehead's quadratic functor $\Gamma$.
%, which is a functor from the category of abelian groups to itself.
As part of our proof of the main result, we give a precise homological description  of the kernel of the natural map
$\Gamma(A) \to A\otimes_{\mathbb{z}} A$, $\gamma(a)\mapsto a\otimes a$.

\end{abstract}
\maketitle

%%%%%%%%%%%%%%%%%%%%%%%%%%%%%%%%%%%%%%%%
\section*{Introduction}
%%%%%%%%%%%%%%%%%%%%%%%%%%%%%%%%%%%%%%%%
\blfootnote{{\sf 2010 Mathematics Subject Classification:} 20J06, 55T10}
\blfootnote{{\sf Key Words:}  Homology of groups, stem-extensions, Eilenberg-MacLane spaces}
\blfootnote{{\sf Affiliation:} Institute of Mathematics and Computer Sciences (ICMC),
University of Sao Paulo (USP), Sao Carlos, Brazil.

E-mails:\ bmirzaii@icmc.usp.br, 

\ \ \ \ \ \ \ \ \ \ \   f.mokari61@gmail.com, 

\ \ \ \ \ \ \ \ \ \ \  d.carbajalordinola@gmail.com}

It is usually difficult to calculate the (integral) homologies of a given group $G$. One way to deal with this problem
is to take a normal subgroup $N$ of $G$ and study the  Lyndon-Hochschild-Serre spectral sequence associated to the
extension $N\tail G \two G/N$,  which ties up the homologies of these groups:
\[
E_{p,q}^2=H_p(G/N, H_q(N,\z))\Rightarrow H_{p+q}(G,\z).
\]
This is a powerful tool.  There are many interesting cases that the homologies of $N$ and $G/N$ are 
more accessible, and the above spectral sequence can be used to study the homologies of $G$.

The natural  homology maps $H_n(N,\z) \arr H_n(G,\z)$
and $H_n(G,\z)\arr H_n(G/N,\z)$ are hidden, in a way,  in the above spectral sequence. For example
$H_n(N,\z) \arr H_n(G,\z)$ is the composite $H_n(N,\z) \two E_{0,n}^\infty \harr H_n(G,\z)$.

In this article we will study the image of the map $H_3(A,\z) \arr H_3(G,\z)$, when $A\tail G\two Q$
is a quasi stem-extension. We say that a central extension is a quasi stem-extension if the maps
$\tor_i^\z(A, A)\arr \tor_i^\z(A, H_1(G,\z))$, induced by $A\arr G$, are trivial for $i=0,1$. For example perfect central 
extensions are quasi stem-extensions.

As our main theorem (Theorem \ref{2-torsion-image}) we show that when $A\tail G\two Q$ is a quasi 
stem-extension, then there is a natural map 
\[
\varphi: H_1(\Sigma_2^\epsilon, \tors({}_{2^\infty}A,{}_{2^\infty}A))\arr H_3(G,\z)/\rho_\ast(A \otimes_\z H_2(G,\z)) 
\]
such that the image of $\varphi$ coincides with the natural image of $H_3(A,\z)$ in $H_3(G,\z)/\rho_\ast(A \otimes_\z H_2(G,\z)) $.
In particular if the extension is a universal central extension, then the image of $H_3(A,\z)$ in $H_3(G,\z)$ coincides
with the image of $\varphi$, which is $2$-torsion. In the above map ${}_{2^\infty}A$ is the 2-power torsion subgroup of $A$,  
$\Sigma_2:=\{\id, \sigma^\varepsilon \}$ is the symmetric group with two elements and $\sigma^\varepsilon$ being the involution 
on $\tors({}_{2^\infty}A,{}_{2^\infty}A)$ induced by the involution $A\times A \arr A\times A$, $(a, b)\mapsto (b, a)$.

An important tool to investigate our main problem is Whitehead's quadratic functor. This functor, which first appeared in the context
of algebraic topology, is a functor from the category of abelian groups to itself  and usually is denoted by $\Gamma$. Most of important 
aspects of this functor is known and it has been generalised in various ways. 

For an abelian group $A$, we give a precise homological description of the kernel of the 
natural map
\[
\Gamma(A) \arr A\otimes_\z A,\ \ \ \ \ \gamma(a) \mapsto a\otimes a,
\]
which it is known to be $2$-torsion. 
%(The cokernel of this map is isomorphism to $H_2(A,\z)$, 
%the second integral homology group of $A$.)
%In this short article we give a precise homological  description of the kernel of the above map. 
In fact we prove that (Theorem \ref{H4}) we have  the exact sequence
\[
0\arr H_1(\Sigma_2^\varepsilon, \tors({}_{2^\infty}A,{}_{2^\infty}A)) \arr 
\Gamma(A) \arr A\otimes_\z A \arr H_2(A,\z)\arr 0.
\]
%~\\
%{\bf Notation.} If $A \arr B$ is a homomorphism of abelian groups, by $B/A$ we mean $\coker(A \arr
%B)$. For a group $A$, ${}_{n}A$ is the subgroup of  $n$-torsion elements of $A$.
%For prime $p$, ${}_{p^\infty}A$ is the $p$-power torsion subgroup of $A$. 

%%%%%%%%%%%%%%%%%%%%%%%%%%%%%%%%%%%%%%%
\section{Whitehead's quadratic functor}
%%%%%%%%%%%%%%%%%%%%%%%%%%%%%%%%%%%%%%%

A function $\psi: A \arr B$ of (additive) abelian groups is called a quadratic map if
\par (1) for any $a \in A$, $\psi(a)=\psi(-a)$,
\par (2) the function $A \times A \arr B$ with
$(a,b) \mapsto \psi(a+b)-\psi(a)-\psi(b)$ is bilinear.

For any abelian group $A$, there is a universal quadratic map 
\[
\gamma: A \arr \Gamma(A)
\]
such that for any quadratic map $\psi: A \arr B$,
there is a unique group homomorphism $\Psi: \Gamma(A) \arr B$ such that
$\Psi\circ \gamma=\psi$. It is easy to see that $\Gamma$ is a functor 
from the category of abelian groups to itself.

The functions $\phi: A \arr A/2$ and $\psi: A \arr A \otimes_\z A$, given by
$\phi(a)=\bar{a}$ and $\psi(a)=a\otimes a$ respectively, are quadratic maps.
Thus we get the canonical homomorphisms
\[
\Phi: \Gamma(A) \arr A/2,\  \gamma(a) \mapsto \bar{a}
\ \ \ \  \text{and} \ \ \ \  
\Psi:\Gamma(A) \arr A\otimes_\z A,\ \gamma(a) \mapsto a\otimes a.
\] 
Clearly $\Phi$ is surjective and $\coker(\Psi)=A \wedge A\simeq H_2(A,\z)$. 
Furthermore we have the bilinear pairing 
\[
[\ , \ ]: A \otimes_\z A \arr \Gamma(A), 
\ \ [a,b]:=\gamma(a+b)-\gamma(a)-\gamma(b).
\]
It is easy to see that for any $a,b,c \in A$, $[a,b]=[b,a]$, $\Phi[a,b]=0$, 
$\Psi[a,b]=a\otimes b + b \otimes a$ and $[a+b,c]=[a,c]+[b,c]$.
Using (1) and this last equation, for any $a,b,c \in A$, we obtain
\par (a) $\gamma(a)=\gamma(-a)$,
\par (b) $\gamma(a+b+c)-\gamma(a+b)-\gamma(a+c)-\gamma(b+c)
+\gamma(a)+\gamma(b)+\gamma(c)=0$.

Using these properties we can construct $\Gamma(A)$.  Let $\A$ be the free
abelian group generated by the symbols $w(a)$, $a \in A$. Set 
$\Gamma(A):=\A/\RR$, where $\RR$ denotes the relations (a) and (b) with
$w$ replaced by $\gamma$. Now $\gamma:A \arr \Gamma(A)$ is given by
$a \mapsto \overline{w(a)}$. 

Using this properties one can show that for any nonnegative integer $n$, we have
\[
\gamma(na)=n^2\gamma(a).
\]

It is known that the sequence  
\[
A \otimes_\z A \overset{[\ , \ ]}{\larr} \Gamma(A) \overset{\Phi}{\arr} A/2 \arr 0
\]
is exact and the kernel of $[\ ,\ ]$ is generated by the elements of the form 
$a\otimes b -b \otimes a$, $a, b \in A$. Therefore we have the exact sequence 
\begin{equation}\label{exact1}
0\arr H_0(\Omega_2, A\otimes_\z A) \overset{[\ , \ ]}{\larr} \Gamma(A) 
\overset{\Phi}{\arr} A/2\arr 0,
\end{equation}
where $\Omega_2:=\{\id, \omega\}$ and $\omega$ is the involution
$\omega(a\otimes b)=b\otimes a$ on $A\otimes_\z A$.

It is easy to see that the composition
\[
A \otimes_\z A \overset{[\ , \ ]}{\larr} \Gamma(A) \overset{\Psi}{\arr} A \otimes_\z A
\]
takes $a\otimes b$ to $a\otimes b+b\otimes a$. Moreover  the composition 
\[
\Gamma(A) \overset{\Psi}{\arr} A \otimes_\z A \overset{[\ , \ ]}{\larr} \Gamma(A) 
\]
coincide with multiplication by $2$. Thus $\ker(\Psi)$ is 2-torsion. 

To give a homological description of the kernel of $\psi$, we will need the following fact.

\begin{prp}\label{KA2}
For any abelian group $A$,  $\Gamma(A)\simeq H_4(K(A,2),\z)$, where
$K(A,2)$ is the Eilenberg-Maclane space of type $(A,2)$. 
\end{prp}
\begin{proof}
See \cite[Theorem 21.1]{eilenberg-maclane1954}
\end{proof}

%%%%%%%%%%%%%%%%%%%%%%%%%%%%%%%%%%%%%%%%%%%%%
\section{Tor-functor and third homology of abelian groups}
%%%%%%%%%%%%%%%%%%%%%%%%%%%%%%%%%%%%%%%%%%%%%

Let $A$ and $B$ be abelian groups. For any positive integer $n$ there is a natural homomorphism
\[
\tau_n:{}_nA\otimes_\z {}_nB \arr {}_n\tors (A,B).
\]
We denote the image of $a\otimes b$, under $\tau_n$ by $\tau_n(a,b)$.

For any pair of integers $s$ and $n$ such that $n=sm$, the maps $\tau_n$ are related by the 
commutative diagrams
\[
\begin{tikzcd}
 &{}_nA\otimes_\z {}_s B \ar[dl, " p_m\otimes \id"]\ar[dr, "\id\otimes i_m "] &\\
{}_sA \otimes_\z {}_s B \ar[dr, " \tau_s"] & &{}_nA \otimes_\z {}_nB,\ar[dl, "\tau_n "]\\
&{}_n\tors (A,B) &
\end{tikzcd}
\]
\[
\begin{tikzcd}
 &{}_sA\otimes_\z {}_n B \ar[dl, " \id\otimes p_m"]\ar[dr, "i_m\otimes \id "] &\\
{}_sA \otimes_\z {}_s B \ar[dr, " \tau_s"] & &{}_nA \otimes_\z {}_nB,\ar[dl, "\tau_n "]\\
&{}_n\tors (A,B) &
\end{tikzcd}
\]
in which $i_m:{}_sA \arr {}_nA$ and $p_m:{}_nA\arr {}_sA$ are the inclusion and 
the map induced by multiplication by $m$ respectively.
The commutativity of these diagrams expresses the relations
\[
\tau_n(a,b)=\tau_s(ma,b), \ \ \ \text{for $a\in {}_nA$ and $b\in {}_sB$},
\]
and 
\[
\tau_n(a',b')=\tau_s(a',mb'), \ \ \ \text{for $a'\in {}_sA$ and $b'\in {}_nB$}.
\]
The following proposition is well-known \cite[Proposition 3.5]{breen1999}.

\begin{prp}\label{tor-lim}
The induced map $\tau: \lim_{I}(_nA\otimes {}_n B) \arr \tors(A, B)$, where $I$ is the inductive system of 
objects ${}_nA\otimes_\z {}_n B$ determined by the above diagrams for varying $n$, is an isomorphism.
\end{prp}

Let $\sigma_0: A\otimes B\arr B\otimes A$ and $\sigma_1:\tors (A,B)\arr \tors (B,A)$
be induced by interchanging the groups $A$ and $B$. It is well known that the diagram
\[
\begin{tikzcd}
{}_nA\otimes_\z {}_nB \ar[r," \sigma_0"]\ar[d, "\tau_n "] & {}_nB \otimes_\z {}_nA \ar[d, "\tau_n' "]\\
{}_n\tors (A,B)\ar[r, " -\sigma_1"] &{}_n\tors (B,A)
\end{tikzcd}
\]
commutes.
By passing to the inductive limit, the same is true for the diagram
\[
\begin{tikzcd}
\lim_I({}_nA\otimes_\z {}_nB) \ar[r," \sigma_0"]\ar[d, "\tau "] & \lim_I({}_nB \otimes_\z {}_nA) \ar[d, "\tau' "]\\
\tors (A,B)\ar[r, " -\sigma_1"] &\tors (B,A).
\end{tikzcd}
\]
It is useful to observe that the map $\sigma_1:\tors (A,B)\arr \tors (B,A)$ is indeed induced by 
the involution $A\otimes_\z B \arr B\otimes_\z A$ given by $a\otimes b\mapsto -b\otimes a$
and therefore $-\sigma_1$ is induced by the involution $a\otimes b\mapsto b\otimes a$

Let $\Sigma_2$ be the symmetric group of order 2. For an abelian group $A$, $\Sigma_2$ acts on 
$A\otimes_\z A$ and $\tors(A,A)$, through $\sigma_0$ and $\sigma_1$.
Let us denote the symmetric group by $\Sigma_2^\varepsilon$, 
rather than simply by $\Sigma_2$, when it
acts on $\tors(A,A)$ as 
\[
(\sigma^\varepsilon, x)\mapsto -\sigma_1(x).
\]
%In fact this action is induced by 
%the involution $A\otimes_\z B \arr B\otimes_\z A$ given by $a\otimes b\mapsto -b\otimes a$.

%\begin{cor}\label{invariant}
%Let $A$ be an abelian group. The the isomorphism $\tau: \lim_{I}(_nA\otimes {}_n A) \arr \tors(A, A)$ induces the
%isomorphism
%\[
%\lim_{I}(_nA\otimes {}_n A)^{\Sigma_2} \overset{\tau}{\simeq} \tors(A, A)^{\Sigma_2^\varepsilon}.
%\]
%\end{cor}
We need the following well-known lemma on the third homology of abelian groups
 \cite[Lemma~5.5]{suslin1991}, \cite[Section 6]{breen1999}. 

\begin{prp}\label{H3A}
For any abelian group $A$ we have the exact sequence
\[
\begin{array}{c}
0 \arr \bigwedge_\z^3 A \arr H_3(A,\z) \arr \tors(A,A)^{\Sigma_2^\varepsilon} \arr 0,
\end{array}
\]
where the right side homomorphism  is obtained from the composition
\[
H_3(A,\z) \overset{{\Delta_A}_\ast}{\larr } H_3(A\times A,\z) \arr \tors(A,A),
\]
$\Delta_A$ being the diagonal map $A \arr A\times A$, $a \mt (a,a)$.
\end{prp}
%\begin{proof}
%Since the direct limit commutes with the wedge product and the homology 
%we may assume that $A$ is finitely generated. In this case the claims 
%are easy to prove. 
%See \cite[Lemma~5.5]{suslin1991}.
%\end{proof}

%%%%%%%%%%%%%%%%%%%%%%%%%%%%%%%%%%%%%%%%%%%%%%%
\section{The kernel of  \texorpdfstring{$\Psi: \Gamma(A) \arr A\otimes A$}{Lg}}
%%%%%%%%%%%%%%%%%%%%%%%%%%%%%%%%%%%%%%%%%%%%%%%

We study the kernel of $\Psi:\Gamma(A)\arr A\otimes_\z A$. If $\Theta=[\ ,\ ]:A\otimes_\z A\arr \Gamma(A)$, then
from the commutative diagram
\[
\begin{tikzcd}
0\ar[r] &\ker(\Theta) \ar[r]\ar[d] & A\otimes_\z A \ar[r, "\Theta"]\ar[d, "\Theta"] & \im(\Theta) \ar[d, "\gamma"] \lar[r]&0\\
0\ar[r] &\ker(\Psi) \ar[r]& \Gamma(A)\ar[r, "\gamma"] &A\otimes_\z A & 
\end{tikzcd}
\]
and exact sequence (\ref{exact1}) we obtain the exact sequence
\[
\ker(\Psi) \arr A/2 \overset{\delta}\arr (A\otimes_\z A)_{\Omega_2} \arr H_2(A,\z) \arr 0,
\]
where $(A\otimes_\z A)_{\Omega_2}=(A\otimes_\z A)/ \langle a\otimes b + b\otimes a | a,b \in A\rangle$
and $\delta(\overline{a})= \overline{a\otimes a}$.
But the sequence 
\[
0\arr A/2 \arr (A\otimes_\z A)_{\Omega_2} \arr H_2(A,\z) \arr 0
\]
is always exact. Thus the map $\ker(\Psi) \arr A/2$ is trivial, which shows that 
\[
\ker\Big(\Gamma(A)\overset{\Psi}{\larr} A\otimes_\z A\Big)\se 
\im\Big(A\otimes_\z A \overset{[\ ,\ ]}{\larr} \Gamma(A)\Big).
\]

We give a precise description of the kernel of $\Psi$.

\begin{thm}\label{H4}
For any abelian group $A$, we have the exact sequence
\[
0\arr H_1(\Sigma_2^\varepsilon, \tors({}_{2^\infty}A,{}_{2^\infty}A)) \arr 
\Gamma(A) \overset{\Psi}{\arr} A\otimes_\z A \arr H_2(A,\z)\arr 0.
\]
\end{thm}
\begin{proof}
If $A \tail B \two C$ is an extension of abelian groups, then
standard classifying space theory gives a (homotopy theoretic) fibration
of Eilenberg-MacLane spaces $K(A,1) \arr K(B,1) \arr K(C,1)$.
From this we obtain the fibration \cite[Lemma 3.4.2]{may-ponto2012}
\[
K(B,1) \arr K(C,1) \arr K(A,2).
\]

For the group $A$, the morphism of extensions
\[
\begin{tikzcd}
A \ar[d]\ar[r, tail, "i_1"] & A \times A \ar[d, "\mu"] 
\ar[r, two heads, "p_2"] & A \ar[d]\\\
A \ar[r, tail, "="]  & A  \ar[r, two heads] & \{1\},
\end{tikzcd}
\]
where $i_1(a)=(a,1)$, $p_2(a,b)=b$ and $\mu(a,b)=ab$, induces the morphism of
fibrations
\[
\begin{tikzcd}
K(A \times A, 1) \ar[r] \ar[d] & K(A,1) \ar[r] \ar[d] & K(A,2)\ar[d]\\
K(A, 1) \ar[r] & K(\{1\},1) \ar[r] & K(A,2).
\end{tikzcd}
\]
By analysing the Serre spectral sequences associated to this morphism of 
fibrations, we obtain the exact sequence
\[
0\arr  \ker(\Psi) \arr H_4(K(A,2)) \overset{\Psi}{\arr} 
A\otimes_\z A \arr H_2(A)\arr 0,
\]
where
\[
\ker(\Psi)\simeq H_3(A,\z)/\mu_\ast(A\otimes_\z H_2(A,\z) \oplus \tors(A,A)).
\]
By Proposition  \ref{H3A} we have the exact sequence
\[
\begin{array}{c}
0 \arr \bigwedge_\z^3 A \arr H_3(A,\z) \arr \tors(A,A)^{\Sigma_2^\varepsilon} \arr 0.
\end{array}
\]
Clearly $\mu_\ast(A\otimes_\z H_2(A,\z))\se \bigwedge_\z^3 A$.
Therefore
\[
\ker(\Psi)\simeq \tors(A,A)^{\Sigma_2^\varepsilon}/{(\Delta_A\circ\mu)}_\ast(\tors(A,A)).
\]
We prove that the map 
$\Delta\circ \mu:A\times A \arr A \times A$, which is given by
$(a,b)\mapsto (ab,ab)$, induces the map
\[
\id+\sigma^\varepsilon: \tors(A,A) \arr \tors(A,A).
\]
By studying the map 
$(\Delta\circ \mu)_\ast: H_2(A\times A,\z) \arr H_2(A\times A,\z)$ using the fact that
$A\otimes A\simeq H_2(A\times A,\z)/(H_2(A,\z)\oplus H_2(A,\z))$ (the K\"unneth Formula),
one sees that $\Delta\circ \mu$ induces the map
\[
A\otimes A \arr A \otimes A, \ \ \ a\otimes b \mapsto a\otimes b-b\otimes a,
\] 
Thus to study the induced map on $\tors(A,A)$ by  $\Delta\circ \mu$ we should study
the map induced on $\tors(A,A)$ by the map 
\[
A\otimes A \arr A \otimes A, \ \ \ a\otimes b \mapsto a\otimes b+b\otimes a=(\id+\iota)(a\otimes b),
\]
where $\iota: A\otimes A \arr A \otimes A$ is given by $a\otimes b\mt b\otimes a$. 
Let
\[
0 \arr F_1 \overset{\partial}{\larr} F_0 \overset{\epsilon}{\larr} A \arr 0
\]
be a free resolution of $A$. Then the sequence
\[
0 \arr F_1 \otimes F_1 
\overset{\partial_2}{\larr} 
F_0 \otimes F_1 \oplus F_1 \otimes F_0 
\overset{\partial_1}{\larr} 
F_0 \otimes F_0
%\overset{\partial_0^\otimes}{\larr} A \otimes A 
\arr 0
\]
can be used to calculate $\tors(A,A)$, where
$\partial_2=(\partial\otimes \id_{F_1}, -\id_{F_1}\otimes \partial)$,
$\partial_1=\id_{F_0}\otimes \partial +\partial\otimes \id_{F_0}$. The map
$\id+\iota: A\otimes A \arr A\otimes A$ can be extended to the morphism of complexes
\[
\begin{tikzcd}
0 \larr  \!\!\!\!\!\!\!\!\!\!\!\!\!\!\!\! &F_1 \otimes F_1 \ar[r, "\partial_2"] 
\ar[d, "f_2"] & 
F_0 \otimes F_1 \oplus F_1 \otimes F_0 \ar[r,"\partial_1"] \ar[d, "f_1"] & 
F_0 \otimes F_0 \ar[d, "f_0"]& \!\!\!\!\!\!\!\!\!\!\!\!\!\!\!\!\larr 0 \\
0 \larr  \!\!\!\!\!\!\!\!\!\!\!\!\!\!\!\! &F_1 \otimes F_1 \ar[r, "\partial_2"] & 
F_0 \otimes F_1 \oplus F_1 \otimes F_0 \ar[r,"\partial_1"]& 
F_0 \otimes F_0 & \!\!\!\!\!\!\!\!\!\!\!\!\!\!\!\!\larr 0,
\end{tikzcd}
\]
where
\begin{align*}
&f_0(x\otimes y):= x\otimes y +y\otimes x, \\
&f_1(x\otimes y, y'\otimes x'):=(x\otimes y+x'\otimes y', y\otimes x+y'\otimes x'),\\
&f_2(x\otimes y):=x\otimes y-y\otimes x. 
\end{align*}
Since 
\[
f_1(x\otimes y, y'\otimes x')=(x\otimes y, y'\otimes x')+(x'\otimes y', y\otimes x),
\]
$\Delta\circ \mu$ induces the map $\id +\sigma^\varepsilon:\tors(A,A) \arr \tors(A,A)$.
Therefore
\[
\ker(\Psi)\simeq \tors(A,A)^{\Sigma_2^\varepsilon}/(\id+\sigma^\varepsilon)(\tors(A,A))=
H_1(\Sigma_2^\varepsilon, \tors(A,A)).
\]

Finally since $\tors(A,A)=\tors(A_{\rm T},A_{\rm T})$, $A_{\rm T}$ being
the subgroup of torsion elements of $A$, and since for any torsion abelian 
group $B$, $B \simeq \bigoplus_{p \ {\rm prime}} {}_{p^\infty}B$,
we have the isomorphism
\[
H_1(\Sigma_2, \tors(A,A))\simeq H_1(\Sigma_2, \tors({}_{2^\infty}A,{}_{2^\infty}A)).
\]
This completes the proof of the theorem.
\end{proof}

\begin{cor}
For any abelian group $A$, we have the exact sequence
\[
\begin{array}{c}
0\arr \lim_{I}H_1(\Sigma_2,{}_{2^n}A\otimes_\z {}_{2^n} A) \arr 
\Gamma(A) \overset{\Psi}{\arr} A\otimes_\z A \arr H_2(A,\z)\arr 0.
\end{array}
\]
In particular if ${}_{2^\infty}A$ is finite then we have the exact sequence
\[
\begin{array}{c}
0\arr H_1(\Sigma_2,{}_{2^\infty}A\otimes_\z {}_{2^\infty} A) \arr 
\Gamma(A) \overset{\Psi}{\arr} A\otimes_\z A \arr H_2(A,\z)\arr 0.
\end{array}
\]
\end{cor}
\begin{proof}
This follows from Theorem \ref{H4} and Proposition \ref{tor-lim}.
\end{proof}

%%%%%%%%%%%%%%%%%%%%%%%%%%%%%%%%%%%%%%%%%%%%
\section{The third homology of stem-extensions}\label{stem}
%%%%%%%%%%%%%%%%%%%%%%%%%%%%%%%%%%%%%%%%%%%%

A central extension $A \tail G \two Q$ is called a {\it stem-extension} if
the natural map $A=H_1(A,\z) \arr H_1(G,\z)$ is trivial and it is called a {\it weak
stem-extension} if the natural map $A \otimes A \arr A \otimes_\z H_1(G,\z)$ is trivial \cite{stammbach1973}. 
%The extension is called an {\it stem cover} if the natural map $H_2(Q) \arr A$ is an isomorphism.

 Let $A \tail G \two Q$ be a weak stem-extension and let $n\geq 2$. From the commutative diagram 
\begin{equation}\label{diagram-mul}
\begin{tikzcd}
A \times A \ar[r, "\mu"]\ar[d] & A \ar[d]\\
A \times G \ar[r, "\rho"] & G,
\end{tikzcd}
\end{equation}
where $\mu$ and $\rho$ are the usual multiplications, we obtain
the commutative diagram
\[
\begin{tikzcd}
H_{n-2}(A,\z)\otimes_\z H_1(A,\z) \otimes_\z H_1(A,\z) \ar[r] \ar[d, "=0"] & H_n(A,\z) \ar[d]\\
H_{n-2}(A,\z)\otimes_\z H_1(A,\z) \otimes_\z H_1(G,\z) \ar[r] & H_{n}(G,\z).
\end{tikzcd}
\]
It follows that the composite 
$\begin{array}{c}
\bigwedge_\z^n A \arr H_n(A,\z) \arr H_n(G,\z)
\end{array}$
is trivial. In particular, the natural image of $H_n(A,\z)$ in $H_n(G,\z)$ is torsion, because 
the map $\bigwedge_\z^n A \arr H_n(A,\z)$ has torsion cokernel \cite[Theorem 6.4, Chap. V]{brown1994}.

We say that a central extension $A\tail G \two Q$ is a {\it quasi stem-extension}
if the natural maps $\tor_i^\z(A,A) \arr \tor_i^\z(A, H_1(G,\z))$  are trivial for $i=0,1$.
In particular quasi stem-extensions are weak stem extensions. The following theorem  
generalizes \cite[Theorem]{M-M-O}.

\begin{thm}\label{2-torsion-image}
Let $A\tail G\two Q$ be a quasi stem-extension.  Then there is a natural map 
\[
\varphi: H_1(\Sigma_2^\epsilon, \tors({}_{2^\infty}A,{}_{2^\infty}A))\arr H_3(G,\z)/\rho_\ast(A \otimes H_2(G,\z)) 
\]
such that the image of $\varphi$ coincides with the natural image of $H_3(A,\z)$ in $H_3(G,\z)/\rho_\ast(A \otimes H_2(G,\z)) $. 
In particular the image of $H_3(A,\z)$ in $H_3(G,\z)/\rho_\ast(A \otimes H_2(G,\z)) $ is $2$-torsion.
\end{thm}
\begin{proof}
We argue as the proof of \cite[Theorem 2]{M-M-O}.
By Proposition \ref{H3A} we have  the exact sequence
\[
\begin{array}{c}
0 \arr \bigwedge_\z^3 A \arr H_3(A,\z) \arr \tors(A,A)^{\Sigma_2^\epsilon} \arr 0,
\end{array}
\]
where the homomorphism on the right side of the exact sequence is obtained from the composition
$H_3(A,\z) \overset{{\Delta}_\ast}{\larr } H_3(A\times A,\z) \arr \tors(A,A)$,
$\Delta$ being the diagonal map $A \arr A\times A$, $a \mt (a,a)$. 
%%Moreover the action 
%%of $\sigma^\epsilon$ on  $ \tors(A,A)$ is induced by the involution $\iota: A \times A \arr A \times A$, $(a, b) \mt (b, a)$.
From (\ref{diagram-mul}), we obtain the commutative diagram
\[
\begin{tikzcd}
\widetilde{H}_3(A\times A)\ar[r, "\mu_\ast"]\ar[d]& H_3(A,\z)\ar[d]\\
\widetilde{H}_3(A\times G)\ar[r, "\rho_\ast"]& H_3(G,\z),
\end{tikzcd}
\]
where
\[
{\widetilde{H}}_3(A\times A):=\ker(H_3(A\times A,\z)\overset{({p_1}_\ast, {p_2}_\ast)}{-\!\!\!-\!\!\!-\!\!\!\larr}
H_3(A,\z) \oplus H_3(A,\z)),
\]
\[
\tilde{H}_3(A\times G):=\ker(H_3(A\times G,\z)\overset{({p_1}_\ast, {p_2}_\ast)}{-\!\!\!-\!\!\!-\!\!\!\larr}
H_3(A,\z) \oplus H_3(G,\z)).
\]
The previous lemma  implies that  the composite 
\[
\begin{array}{c}
\bigwedge_\z^3 A\arr H_3(A,\z)\arr H_3(G,\z)
\end{array}
\]
is trivial. These facts together with the K\"unneth formula  give us the commutative diagram
\[
\begin{tikzcd}
\tors(A,A) \ar[ddd, bend right=300] \ar[r, "\bar{\mu}_\ast"] & 
\tors(A,A)^{\Sigma_2^\epsilon}\\
{\tilde{H}}_3(A\times A) /\bigoplus_{i=1}^2 H_i(A,\z)\otimes_\z H_{3-i}(A,\z)
\ar[ u, "\overset{\alpha}{\simeq}"]\ar[r, "\mu_\ast"] \ar[d, , "\widetilde{\inc}_\ast"]& 
H_3(A,\z)/\bigwedge_\z^3 A \ar[u, "\overset{\beta}{\simeq}"] \ar[d, "\inc_\ast"] \\
{\tilde{H}}_3(A\times G)/\bigoplus_{i=1}^2 H_i(A,\z)\otimes_\z H_{3-i}(G,\z) \ar[d, "\simeq"] \ar[r, "\rho_\ast"]& 
H_3(G,\z)/M\\
\tors(A,H_1(G,\z)), &
\end{tikzcd}
\]
where $M=\rho_\ast(A\otimes_\z H_2(G,\z))$. Note that $\im(H_2(A,\z)\otimes_\z H_1(G,\z) \arr H_3(G,\z))\se M$
(see \cite[Proposition 4.4, Chap. V]{stammbach1973}).
Since the map 
\[
\tors(A,A)\arr \tors(A,H_1(G,\z))
\]
is trivial,  $\rho_\ast\circ \widetilde{\inc}_\ast\circ \alpha^{-1}=0$. This shows that the 
composite map $\inc_\ast\circ\beta^{-1}\circ \bar{\mu}_\ast$ is trivial. Therefore the image of $H_3(A,\z)$ in 
$H_3(G,\z)/M$ is equal to the image of
\[
\tors(A,A)^{\Sigma_2^\epsilon}/\bar{\mu}_\ast\tors(A,A)=\tors(A,A)^{\Sigma_2^\varepsilon}/{(\Delta_A\circ\mu)}_\ast(\tors(A,A)).
\]
Now as in the proof of Theorem \ref{H4}, one can show that 
\[
\tors(A,A)^{\Sigma_2^\varepsilon}/{(\Delta_A\circ\mu)}_\ast(\tors(A,A))\simeq \ker(\Psi)
\]
which by Theorem \ref{H4} is isomorphic to $H_1(\Sigma_2^\epsilon, \tors({}_{2^\infty}A,{}_{2^\infty}A))$. 
This complete the proof of the theorem.
\end{proof}

\begin{rem}
In general $\varphi$ does not necessarily factors through $\Gamma(A)$. In fact we have the following commutative diagram
\[
\begin{tikzcd}
H_1(\Sigma_2^\epsilon, \tors({}_{2^\infty}A,{}_{2^\infty}A)) \ar[r, "\varphi"] \ar[d, hook] & H_3(G,\z)/\rho_\ast(A \otimes H_2(G,\z))\ar[d, two heads]\\
\Gamma(A) \ar[r]& H_3(G,\z)/\rho_\ast(\tilde{H}_3(A \times G)).
\end{tikzcd}
\]
Thus $\varphi$ factors through $\Gamma(A)$ if the extension is a stem-extension.
\end{rem}

\begin{cor}
Let $A\tail G\two Q$ be a universal central extension.  Then the image of $H_3(A,\z)$ in $H_3(G,\z)$ coincides with the image of
\[
\varphi: H_1(\Sigma_2^\epsilon, \tors({}_{2^\infty}A,{}_{2^\infty}A))\arr H_3(G,\z).
\]
In particular the image of $H_3(A,\z)$ in $H_3(G,\z)$ is $2$-torsion.
\end{cor}

\begin{rem}
Let $A \tail G \two Q$ be a perfect central extension, i.e. $G$ is a perfect group and the extension is central. By an easy analysis
of the Lyndon-Hochschild-Serre spectral sequence of the extension, we obtain
the exact sequence 
\[
H_4(Q,\z) \arr T \arr H_3(G,\z)/N \arr H_3(Q,\z) \arr 0, 
\]
where $T=\ker(A \otimes A \arr H_2(A,\z))=\langle a\otimes a: a\in A\rangle$ and
$N:=\rho_\ast\Big(H_3(A,\z)\oplus A \otimes H_2(G,\z)\Big)$.

Moreover from the fibration of Eilenberg-MacLane spaces $K(A,1) \arr K(G,1) \arr K(Q,1)$,
we obtain the fibration \cite[Lemma 3.4.2]{may-ponto2012}
\[
K(G,1) \arr K(Q,1) \arr K(A,2).
\]
By studying the Serre spectral sequence associated to this fibration
we obtain the exact sequence
\[
H_4(Q, \z)\arr \Gamma(A) \arr H_3(G, \z)/M\arr H_3(Q,\z)\arr 0,
\]
where $M:=\rho_\ast\Big(A\otimes_\z H_2(G,\z)\Big)$.
Now we have the following commutative diagram with exact rows and columns:
\[
\begin{tikzcd}
& & 
H_4(Q,\z) \ar[r, "="] \ar[d] & H_4(Q,\z) \ar[d] &\\
0 \larr \!\!\!\!\!\!\!\!\!\!\!\!\!\!\!\!\!\! & H_1(\Sigma_2^\epsilon, \tors({}_{2^\infty}A,{}_{2^\infty}A)) \ar[r]
& \Gamma(A) \ar[r] \ar[d] &  T \ar[d] &
\!\!\!\!\!\!\!\!\!\!\!\!\!\!\!\!\!\!\!\!\!\!\!\!\!\!\!\!\!-\!\!\!
-\!\!\!-\!\!\!-\!\!\!\larr 0\\
& H_3(A,\z)/\bigwedge_{\z}^3 A \ar[u] \ar[r]
&H_3(G,\z)/M \ar[r] \ar[d] & 
H_3(G,\z)/N \ar[d] & \!\!\!\!\!\!\!\!\!\!\!\!\!\!\!\!\!\!\!\!\larr 0.\\
&&  H_3(Q,\z) \ar[r, "="] \ar[d] & H_3(Q,\z)\ar[d] &\\
&&  0 & 0 &
\end{tikzcd}
\]
%where $M:=\rho_\ast(A\otimes H_2(G))$ and $N=\rho_\ast(H_3(A) \oplus A\otimes H_2(G))$.
\end{rem}

\begin{exa}
For any sequence of abelian groups $A_n$, $n\geq 2$, Berrick and Miller constructed 
a perfect group $Q$ such that $H_n(Q,\z)\simeq A_n$ \cite[Theorem~1]{berrick-miller1992}.

Let $A$ be an abelian group. Choose a perfect 
group $Q$ such that $H_2(Q,\z)\simeq A$ and $H_4(Q,\z)=0$. Then if $A \tail G \two Q$ is the 
universal central extension of $Q$, we have the exact sequence
\[
0\arr \Gamma(A) \arr H_3(G,\z)\arr H_3(Q,\z)\arr 0.
\]
Thus the map $\varphi: H_1(\Sigma_2^\epsilon, \tors({}_{2^\infty}A,{}_{2^\infty}A)) \arr H_3(G,\z)$ is injective here.
This example also show that the kernel of the natural map $H_3(G,\z)\arr H_3(Q,\z)$ usually is larger than the image of $\varphi$.
\end{exa}

%\begin{rem}
%Let $A \tail G \two Q$ be a universal central extension. In \cite[Proposition 3.3]{M-M-O}, we show that if the plus construction of the space 
%$K(Q,1)$ is an $H$-space, then the image of $H_3(A,\z)$ in $H_3(G,\z)$ is trivial. In particular the map $\varphi$
%is trivial. There are many of such extensions coming from Algebraic $K$-theory  \cite{loday1976},  \cite{berrick1982}.
%\end{rem}

%%%%%%%%%%%%%%%%%%%%%%%%%%%%%%%%%%%%%%%%%%%%%%%%%

%\bigskip
%\address{{\footnotesize
%
%Institute of Mathematics and Computer Sciences (ICMC),
%
%University of Sao Paulo (USP), Sao Carlos, Brazil.
%
%e-mails:\ bmirzaii@icmc.usp.br,

%\ \ \ \ \ \ \ \ \ \ \  f.mokari61@gmail.com,
%
%\ \ \ \ \ \ \ \ \ \ \  davidcarbajal@usp.br 
%
%}}

\begin{thebibliography}{99}
%%%%%%%%%%%%%%%%%%%%%%%%%%%%%%%%%%%%%%%%%%%%%%%%%

%\bibitem{andre1965}
%Andr\'e, M. Le $d_2$ de la suite spectrale en cohomologie des groupes.
%C. R. Acad. Sci. Paris {\bf 260} (1965), 2669--2671.

%\bibitem{atiyah1969}
%Atiyah, M. F., Macdonald, I. G. Introduction to 
%commutative algebra. Addison-Wesley Publishing Co., 
%Reading, Mass.-London-Don Mills, Ont. 1969.

%\bibitem{berrick1982}
%Berrick, A. J. An Approach to Algebraic K-Theory. Research Notes in Math. 
%No. 56, Pitman, London, 1982. 

%\bibitem{berrick1987}
%Berrick, A. J. Two functors from abelian groups to perfect groups. 
%Proceedings of the Northwestern conference on cohomology of groups 
%(Evanston, Ill., 1985). 
%J. Pure Appl. Algebra {\rm 44} (1987), no. 1-3, 35--43.

\bibitem{berrick-miller1992}
Berrick, A. J., Miller, C. F., III. Strongly torsion generated groups. 
Math. Proc. Cambridge Philos. Soc. {\rm 111} (1992), no. 2, 219--229.

\bibitem{breen1999}
L. Breen, L.On the functorial homology of abelian groups. 
Journal of Pure and Applied Algebra {\bf 142} (1999) 199--237.

\bibitem{brown1994}
Brown, K. S. Cohomology of Groups.  Graduate
Texts in Mathematics, 87. Springer-Verlag, New York, 1994.

%\bibitem{dupont-sah1982}
%Dupont, J. L., Sah, C. Scissors congruences. II. J. Pure Appl. Algebra {\bf 25} 
%(1982), no. 2, 159--195.

%5\bibitem{dupont-parry-sah1988} Dupont, J- L., Parry, W., Sah, C. 
%Homology of classical Lie groups made discrete II.
%$H_2$, $H_3$, and relations with scissors congruences.
%J. Algebra  {\bf 113}  (1988), 215--260.

%\bibitem{eckmann-hiton1971}
%Eckmann, B., Hilton, P. J. On central group extensions and homology.
%Commentarii Mathematici Helvetici {\bf 46} (1971), 345--355.

\bibitem{eilenberg-maclane1954}
Eilenberg, S., MacLane, S. On the groups $H(\Pi, n)$, II: Methods of computation. 
Ann. of Math. {\bf 70} (1954), no. 1, 49--139.

%\bibitem{kuzmin-semenov1998}
%Kuz'min, Yu. V., Semenov, Yu. S. On the homology of a free nilpotent 
%group of class 2. Math. Sb. {\bf 189} (1998), no. 4, 527--560.

%\bibitem{lang2002}
%Lang, S. Algebra, Revised Third Edition, Vol. 211, Graduate
%Texts in Mathematics. Springer-Verlag, New York, 2002.

%\bibitem{loday1976} Loday, J. L. $K$-th\'eorie alg\'ebrique et
%repr\'esentations de groupes. Ann. Sci. \'Ecole Norm. Sup. (4) {\bf 9} 
%(1976), no. 3, 309--377.

%\bibitem{matsumura1989}
%Matsumura, H. Commutative Ring Theory, Cambridge Studies 
%in Advanced Mathematics, 8, Cambridge University Press, Cambridge, 1989.

\bibitem{may-ponto2012}
May, J. P., Ponto, K. More concise algebraic topology: 
Localization, completion, and model categories. Chicago Lectures 
in Mathematics. University of Chicago Press, Chicago, IL, 2012. 

%\bibitem{mikhailov2015}
%Mikhailov, R. Polynomial Functors and Homotopy Theory. In: Franjou V., Touzé A. 
%(eds) Lectures on Functor Homology. Progress in Mathematics, vol 311. Birkh\"auser, 2015

%\bibitem{mirzaii-mokari2016} 
%Mirzaii, B., Mokari, F. Y. Virtual rational Betti numbers of
%nilpotent-by-polycyclic groups. Pacific J. Math. {\bf 283} (2016), 
%no. 2, 381--403.

\bibitem{M-M-O}
Mirzaii, B., Mokari, F. Y., Ordinola, D.C. Third homology of perfect central extensions.
Available at https://arxiv.org/abs/2007.11165

%\bibitem{parry-sah1983}
%Parry, W., Sah, C. Third homology of $\SL(2,\R)$ made discrete.
%J. Pure Appl. Algebra {\bf 30} (1983), 181--209.

%\bibitem{rosenberg1996}
%Rosenberg, J. Algebraic $K$-theory and its application. Graduate text 
%in Mathematics. Springer 147, 1996.

%\bibitem{rotman2009}
%Rotman, J. J. An Introduction to Homological Algebra.
%Universitext, Second Edition, Springer, 2009.

%\bibitem{soule1985} Soul\'e, C. Op\'erations en $K$-th\'eorie alg\'ebrique.
%Canad. J. Math.  {\bf 37}  (1985),  no. 3, 488--550.

\bibitem{stammbach1973}
Stammbach, U. Homology in group theory. Lecture Notes in Mathematics, Vol. 359. 
Springer-Verlag, Berlin-New York, 1973.


\bibitem{suslin1991}
Suslin, A. A. $K\sb 3$ of a field and the Bloch group.
Proc. Steklov Inst. Math. {\bf 183} (1991), no. 4, 217--239.

%\bibitem{wagoner1972}
% J. Wagoner, J. Delooping classifying spaces in algebraic K-theory. 
%Topology {\bf 11} (1972), 349--370.

%\bibitem{weibel1994}
%Weibel, C. A. An Introduction to Homological Algebra.
%Cambridge Studies in Advanced Mathematics, 38. Cambridge
%University Press, Cambridge, 1994.

%\bibitem{weibel2013}
%Weibel, C. A. The $K$-Book: An Introduction to Algebraic $K$-Theory. 
%Graduate Studies in Mathematics, Volume 145, American Mathematical 
%Society, Providence, RI, 2013.

%\bibitem{whitehead1950}
%Whitehead, J. H. C.  A certain exact sequence. Ann. Math. {\bf 52} (1950),
%51--110.

%\bibitem{whitehead1978}
%Whitehead, G. W. Elements of homotopy theory. Springer-Verlag. 1978

%\bibitem{wojtkowiak1985}
%Wojtkowiak, Z. Central extension and coverings, Publ. Sec. Mat. 
%Univ. Aut\'onoma Barcelona {\bf 29} no. 2-3, (1985),  145--153.
%Wojtkowiak, Z. Central extension and coverings.
%Publicacions de la SecciÃ³ de MatemÃ tiques
%{\bf  29}, (1985), no. 2/3, 145--153.

%\bibitem{zabrodsky1976}
%Zabrodsky, A. Hopf spaces, Moth-Holland Publishing Company, 1976.

\end{thebibliography}
\end{document}